\documentclass{amsart}
\usepackage{amssymb}
\usepackage{pdflscape}
\usepackage{amscd}
\usepackage{fancyhdr}
\usepackage[backref=page]{hyperref}
\usepackage{amsmath}
\usepackage{amsthm}
\usepackage{amsfonts}
\usepackage{amssymb}
\usepackage{euscript}
\usepackage[all]{xy}
\usepackage{tikz}
\usepackage{tikz-cd}

\tikzcdset{scale cd/.style={every label/.append style={scale=#1},
		cells={nodes={scale=#1}}}}

\newcommand{\bd}{\mathbf{d}}

\newcommand{\be}{\mathbf{e}}

\newcommand{\degg}{\le_{\rm deg}}
\newcommand{\bdim}{{\rm\bf dim}}

\newcommand{\nc}{\newcommand}

\nc{\la}{\lambda}
\nc{\al}{\alpha }
\nc{\om}{\omega }

\nc{\veps}{\varepsilon}
\nc{\ch}{{\mathop {\rm ch}}}
\nc{\Tr}{{\mathop {\rm Tr}\,}}
\nc{\Id}{{\mathop {\rm Id}}}
\nc{\bra}{\langle}
\nc{\ket}{\rangle}
\nc{\pa}{\partial}
\nc{\ld}{\ldots}
\nc{\cd}{\cdots}
\nc{\hk}{\hookrightarrow}
\nc{\T}{\otimes}
\nc{\mgl}{\mathfrak{gl}}
\nc{\U}{\mathrm U}

\newtheorem{theoremA}{Theorem}

\newtheorem{theorem}{Theorem}[section]

\newtheorem{example}[theorem]{Example}
\newtheorem{definition}[theorem]{Definition}

\numberwithin{equation}{section}
\newtheorem{thm}{Theorem}[section]
\newtheorem{prop}[thm]{Proposition}
\newtheorem{lem}[thm]{Lemma}
\newtheorem{cor}[thm]{Corollary}
\theoremstyle{remark}
\newtheorem{rem}[thm]{Remark}

\newcommand{\bC}{{\mathbb C}}

\newcommand{\bZ}{{\mathbb Z}}

\nc{\Gr}{{\mathop {\rm Gr}}}

\title{Extremality of principal quiver Grassmannians}

\author{Giovanni Cerulli Irelli}
\address{Giovanni Cerulli Irelli:\newline
Sapienza Universit\'a di Roma; Dipartimento di Scienze di Base e Applicate per
l’Ingegneria, Via Antonio Scarpa 14, 00161 Roma, Italy}
\email{giovanni.cerulliirelli@uniroma1.it}

\author{Evgeny Feigin}
\address{Evgeny Feigin:\newline
School of Mathematical Sciences, Tel Aviv University, Tel Aviv
6997801, Israel}
\email{evgfeig@gmail.com}

\author{Markus Reineke}
\address{Markus Reineke:\newline
Ruhr-Universit\"at Bochum, Faculty of Mathematics, Universit\"atsstra{\ss}e 150, 44780 Bochum, Germany}
\email{Markus.Reineke@ruhr-uni-bochum.de}

\begin{document}

\begin{abstract}
We study families of quiver Grassmannians, parametrizing subrepresentations of representations of Dyn\-kin quivers, which are naturally associated to a projective and an injective representation. Our main result is an explicit description of the flat irreducible locus of such a family. More precisely, we prove that a member of this family is irreducible and of the expected dimension if and only if the ambient quiver representation degenerates to the direct sum of the defining projective and injective representations.
\end{abstract}

\maketitle

\section{Introduction}
Let $Q$ be a quiver and let $M$ be a representation of $Q$. Given a dimension 
vector $\be$, we consider the quiver Grassmannian $\Gr_\be(M)$, which parametrizes
subrepresentations of $M$ of dimension vector $\be$. By definition, $\Gr_\be(M)$ is a
projective algebraic variety naturally embedded into a product of classical
Grassmann varieties. Various properties of quiver Grassmannians can be studied
via classical algebro-geometric methods \cite{LW19}, but one can also exploit 
group-theoretic
and algebraic constructions from quiver representation theory (see \cite{CI20,CFR12,Fe23}).
Hence, if a variety admits an explicit realization as a quiver Grassmannian,
one has at hand additional tools for studying its algebro-geometric properties 
(see \cite{CFR13-1,Fe11,FFR17}).

Yet another advantage of a quiver Grassmannian realization is that one immediately
gets a family of varieties, with the initial one being its member. Such families
naturally show up when one varies the ambient $Q$-representation $M$, leading to degenerations
(if $\Gr_\be(M)$ is a general member of the family) or deformations (if $\Gr_\be(M)$ is 
not in general position) of the initial projective variety.

Since every projective variety can be realized as a quiver Grassmannian \cite{Re13,Ri18},
there is no hope to study all quiver Grassmannians together. In this paper we consider what we call the
 PrIncipal case for Dynkin quivers $Q$ (here Pr stands for projective and In 
for injective). More precisely, let $P$ and $I$ be a projective and an injective representation
of $Q$, respectively. We are interested in the varieties of the form $\Gr_{\bdim(P)}(M)$, where 
$\bdim(M) = \bdim(P) + \bdim(I)$; in particular, $\Gr_{\bdim(P)}(P\oplus I)$ is a distinguished 
member of this family. In the equioriented type $A$ case, these varieties for various 
$P$,$I$, and $M$ provide quiver Grassmannian realizations of classical and PBW degenerate 
flag varieties \cite{CFR12,Fe12}; the universal families over the corresponding 
representation spaces (parametrizing representations $M$) were studied in 
\cite{CFFFR17,CFFFR20}.  Since $Q$ is Dynkin, by Gabriel's theorem, there are finitely many isoclasses of representation of $Q$
of dimension vector $\bdim(P) + \bdim(I)$.  One knows \cite{CFR12} that 
for a generic representation
$M$ the quiver Grassmannian $\Gr_{{\rm\bf dim} P}(M)$ is irreducible of dimension 
$[P,I]$, where $[P,I]$ is the dimension of the space of homomorphisms from $P$ to $I$.
For a non-generic $M$, the dimension of $\Gr_{{\rm\bf dim} P}(M)$ can be larger and several irreducible
components may show up. It is thus natural to ask for the description of the flat 
irreducible locus -- the locus of representations $M$ such that $\Gr_{\dim P}(M)$ 
is irreducible and of the minimal possible dimension. 

In order to formulate our  main theorem, let us
introduce the following notation. Let $P_{\rm min}$ be the direct sum of all indecomposable
projectives $P_i$ with $i$ being a source in $Q$. Let $I_{\rm min}$ be the direct sum of 
all indecomposable injectives $I_i$, such that $i$ is a sink in $Q$. Hence,
$P_{\rm min}$ (resp.~$I_{\rm min}$) is the minimal projective (resp.~injective) representation,
whose dimension vector has support $Q$. We prove the following extremality property
of the representations $P\oplus I$.

\begin{theoremA}
Assume that $P_{\rm min}$ is a direct summand of $P$ and $I_{\rm min}$ is a direct summand of $I$. 
Then  the quiver Grassmannian $\Gr_{{\rm\bf dim} P}(M)$ is irreducible of dimension 
$[P,I]$
if and only if $M$ degenerates to $P\oplus I$.
\end{theoremA}

In fact, the condition of the theorem is equivalent to $M$ being an extension of $I$ by $P$. This theorem generalizes the equioriented type $A$ result of \cite{CFFFR17} and proves 
a conjecture of \cite{FedFe25,Fed25}. It shows that the PrIncipal quiver Grassmannian 
${\rm Gr}_{{\rm\bf dim}(P)}(P\oplus I)$ has a natural extremality property, being the 
most degenerate quiver Grassmannian which is irreducible and of the expected dimension.

The main ingredient for the proof is a new category
of $Q$-representations and an explicit description of the complement to the flat irreducible locus.
This category looks interesting and deserves a separate study. We give some details below.

Let $\mathcal{C}$ be the  full subcategory of representations of $Q$ without projective direct 
summands. Let $\mathcal{C}_{\rm simp}$ denote the set of relative simple objects in $\mathcal{C}$. 
In other words, $U\in\mathcal{C}$ belongs to $\mathcal{C}_{\rm simp}$ if and only if it is 
indecomposable, and every proper non-zero subrepresentation of $U$ contains a non-zero projective 
summand. To each $U\in \mathcal{C}_{\rm simp}$ we attach a representation $M_U$ of dimension vector 
$\bdim(P)+\bdim(I)$ and prove the following theorem.

\begin{theoremA}
A quiver Grassmannian $\Gr_{{\rm\bf dim} P}(M)$ is not irreducible or of  dimension larger than 
$[P,I]$ if and only if $M$ is a degeneration of $M_U$ for some 
$U\in \mathcal{C}_{\rm simp}$.
\end{theoremA}

In particular, we show that for each $U\in \mathcal{C}_{\rm simp}$ the representation
$M_U$ does not belong to the flat irreducible locus; note, however, that we do not know if the dimension of
$\Gr_{{\rm\bf dim} P}(M_U)$ is equal to $[P,I]$. In general, it is an interesting and natural
question to describe the flat locus, i.e. the locus of representations $M$ such that
$\dim \Gr_{{\rm \bf dim} P}(M_U)=[P,I]$. We do not have such a description at the moment
(see some conjectures in \cite{FedFe25}).

Recall that if $Q$ is the equioriented type $A$ quiver,
$P$ (resp. $I$) is the direct sum of all projective (resp. injective) indecomposable 
representations, then $\Gr_{{\rm\bf dim} P}(P\oplus I)$ is the PBW degenerate flag variety, and the 
generic member of the family $\Gr_{{\rm\bf dim} P}(M)$ is isomorphic to the classical flag
variety (see e.g. \cite{CFR12}); in particular, for these specific $P$ and $I$,  
the generic quiver Grassmannian is a homogeneous space for
the automorphism group ${\rm Aut}_Q(M)$. The last statement does not hold for general $P$ and $I$ 
(even for the linearly oriented type $A$ quiver).
However, we prove the following theorem.

\begin{theoremA}
Let $Q$ be equioriented type $A$ quiver and let $M$ be a generic $Q$-representations of dimension 
${\rm\bf dim} P + {\rm\bf dim} I$ for arbitrary $P$ and $I$. 
Then the group 
${\rm Aut}_Q(M)$ acts on $\Gr_{{\rm\bf dim} P}(M)$ with an open orbit.
\end{theoremA}

Our paper is organized in the following way. In Section \ref{sec:prelim} we fix notation 
and collect basics from  quiver representation theory.   
In Section \ref{sec:category} we introduce the category $\mathcal{C}$ and almost projective representations. 
In Section \ref{sec:boundrep} we describe the boundary representations separating the flat irreducible locus.
In Section \ref{sec:twostrata} we construct two components of the expected dimension in quiver Grassmannians corresponding to the boundary representations.
In Section \ref{sec:examples} we provide several examples for the objects and constructions 
from the previous sections.
Finally, in Section \ref{sec:prehom} we prove the prehomogeneity for linearly oriented type 
$A$ quivers and provide counterexamples for general
Dynkin quivers.

\section*{Acknowledgments}
We are grateful to Stanislav Fedotov for useful discussions.
EF was partially supported by the ISF grant 493/24.

\section{Preliminaries}\label{sec:prelim}
In this section we briefly recall what we need from the representation theory of quivers. Standard references are \cite{CB92,Schi14}. Let $Q$ be a Dynkin quiver with set of vertices $Q_0$ and set of arrows $Q_1$. For $(\al: i\to j)\in Q_1$, we set $i=s(\alpha)$, $j=t(\al)$.  For us, a representation of $Q$ is a tuple $M=((M_i)_{i\in Q_0}, (M_\al)_{\al\in Q_1})$ where  $M_i$ is a finite-dimensional complex vector space and $M_\al: M_{s(\al)}\to M_{t(\al)}$ is a linear map. The dimension vector of $M$ is $\mathbf{dim} (M)=(\dim\, M_i)_{i\in Q_0}\in \bZ_{\ge0}^{Q_0}$. 
For $i\in Q_0$, we denote by $S_i$ the simple representation supported at vertex $i$, i.e.  $(S_i)_i=\bC$ and $(S_i)_j$ is the zero vector space for every $j\ne i$. We denote by $P_i$ the projective cover of $S_i$ and by $I_i$ its injective envelope. In particular, $P_i\twoheadrightarrow S_i \hookrightarrow I_i$. Recall that $(P_i)_j$ has as a basis the paths in $Q$ from $i$ to $j$ and $(I_i)_j$ has as a basis the paths from $j$ to $i$.  A morphism $f:M\to N$ between two $Q$-representations $M$ and $N$ is a collection $f=(f_i:M_i\to N_i)_{i\in Q_0}$ of linear maps such that $N_\alpha\circ f_i=f_j\circ M_\alpha$. We denote by ${\rm rep}_Q$ the category of $Q$-representations. For a dimension vector  $\bd\in \bZ_{\ge 0}^{Q_0}$, we denote by 
${\rm rep}_Q(\bd)$ the set of isoclasses of $Q$-representations of dimension vector $\bd$. Since $Q$ is Dynkin, by Gabriel's theorem, each set ${\rm rep}_Q(\bd)$ is finite.
 
Let $R_\bd(Q)=\bigoplus_{\al\in Q_1} {\rm Hom}(\bC^{d_{s(\al)}},\bC^{d_{t(\al)}})$ be the 
representation space of dimension vector $\bd$.
The group $G_\bd = \prod_{i\in Q_0} GL_{d_i}$ acts on $R_\bd$ by conjugation, and the orbits
of this group action are in bijection with isomorphism classes of representations of $Q$ of dimension vector $\bd$. 
For two representations $M,N\in R_\bd$, we say that $M$ degenerates to $N$ if  $N\in \overline{G_\bd M}$; in this case we write $M\degg N$. Since $Q$ is a Dynkin quiver, 
$R_\bd(Q)$ admits finitely many $G_\bd$-orbits and, being an affine space and hence irreducible, a unique open dense orbit. It corresponds to a so-called exceptional representation $M^0$ of dimension vector $\bd$, that is, ${\rm Ext}^1_Q(M^0,M^0)=0$. 

For two representations $M,N$, let us denote by $[M,N]$ the dimension of the homomorphism space 
${\rm Hom}_Q(M,N)$ and by $[M,N]^1$ the dimension of the extension space 
${\rm Ext}^1_Q(M,N)$. In particular, the Euler form is given by 
$\langle\dim M,\dim N\rangle = [M,N] - [M,N]^1$.
A theorem of Bongartz \cite{Bo96} says that $M$ degenerates to $N$ if and only if 
\[
[M,X]\le [N,X] \quad  \mbox{ for all } X\in {\rm rep}_Q
\]
or, equivalently, if 
\[
[X,M]\le [X,N]\quad \mbox{ for all } X\in {\rm rep}_Q.
\]

Given two dimension vectors $\bd,\bd'\in\bZ_{\ge 0}^{Q_0}$ and two $Q$-representations $M\in R_\bd(Q)$ and $N\in R_{\bd'}(Q)$, we denote by $M\ast N$ the generic extension of $M$ by $N$ (see \cite[Definition~2.2]{Re01}). This is a $Q$-representation $M\ast N\in R_{\bd+\bd'}(Q)$ which fits into  a short exact sequence 
\begin{equation}\label{eq:GenericExtension}
0\to N\to M\ast N\to M\to 0
\end{equation}
and it has minimal endomorphism ring among all possible extensions of $M$ by $N$ (see \cite[Lemma~2.1]{Re01}). We will make frequent use of the following remarkable property of generic extensions. 
\begin{prop}\cite[Prop.~2.4]{Re01}
The following properties of a representation $X$ are equivalent: 
\begin{itemize}
\item $M\ast N\degg X$;
\item There exists a short exact sequence $0\to N'\to X\to M'\to 0$ such that   $M\degg M'$ and $N\degg N'$.
\end{itemize}
\end{prop}

We will also frequently use the following:

\begin{prop}\label{constremb} \cite[Theorem 2.4]{Bo96}
Suppose that $M$ degenerates to $N$ and that $U$ embeds into $N$. If $[U,M]=[U,N]$, then $U$ 
also embeds into $M$. In particular, this holds for $U$ projective.
\end{prop}

Let $M\in R_\bd(Q)$ be a representation of $Q$. For $\be\in \bZ_{\ge 0}^{Q_0}$ such 
that $e_i\le d_i$
for all $i\in Q_0$, we denote by $\Gr_\be(M)$ the projective variety consisting
of $\be$-dimensional subrepresentations of $M$ and we call it a quiver Grassmannian of 
Dynkin type or simply a quiver Grassmannian
\cite{CI20,CR00}. This variety is the main object of investigation of the present paper. 
We recall that the scheme structure of $\Gr_\be(M)$ is obtained simply by considering the universal 
quiver Grassmannian $\Gr_\be^Q(\bd)$ consisting of collections $(M,N)$, where $M\in R_\bd(Q)$ and 
$N\in \Gr_\be(M)$.
In particular, $\Gr_\be(\bd)$ admits a natural projection to the representation space $R_\bd(Q)$;
the fibers of this projection are the individual quiver Grassmannians. For a representation $N$ 
of dimension vector $\be$, we denote by $\mathcal{S}_{[N]}$ the subset of
${\rm Gr}_{\be}(M)$ of subrepresentations isomorphic to $N$. It is an irreducible locally closed 
subset of dimension $[N,M]-[N,N]$. The $\mathcal{S}_{[N]}$ for various isomorphism classes $[N]$ 
provide a finite stratification of ${\rm Gr}_{\be}(M)$.

Let $P$ be a projective and $I$ be an injective representations of $Q$. 
In this paper we will be interested in the properties of the universal quiver 
Grassmannians $\Gr_\be(\bd)$, where
$\be={\rm\bf dim} P$ and $\bd = {\rm\bf dim}(P\oplus I)$. We have two natural examples of 
such quiver Grassmannians for $M=P\oplus I$ and for $M=M^0$ being a generic representation.
It is shown in \cite{CFR12} that both $\Gr_\be(M^0)$ and $\Gr_\be(P\oplus I)$ are irreducible
of dimension $\langle \be, \bd-\be\rangle = \langle \dim P, \dim I\rangle = [P,I]$.

\begin{example}
Let $Q$ be a linearly oriented type $A_n$ quiver, $P=\bigoplus_{i\in Q_0} P_i$,
$I=\bigoplus_{i\in Q_0} I_i$. Then $\bd=(n+1,\dots,n+1)$, $\be=(1,2,\dots,n)$.
For generic $M$ the quiver Grassmannian $\Gr_\be(M)$ is isomorphic to the classical
flag variety $SL_{n+1}/B$ and for $M=P\oplus I$ the quiver Grassmannian $\Gr_\be(M)$
is the PBW degenerate flag variety. Both distinguished fibers are irreducible and of
dimension $n(n+1)/2$, Moreover, it is proved in \cite{CFFFR17} that $\Gr_\be(M)$
is irreducible of dimension $n(n+1)/2$ if and only if $M$ degenerates to $P\oplus I$.
\end{example}

\section{The category of almost projective representations}\label{sec:category}

Denote by $\mathcal{C}$ the full subcategory of representations of $Q$ without projective direct summands, or equivalently, of representations $M$ such that ${\rm Hom}(M,A)=0$, where $A=\mathbb{C}Q=\bigoplus_{i\in Q_0}P_i$. It is closed under extensions and quotients. Let $\mathcal{C}_{\rm simp}$ denote the set of relative simple objects in $\mathcal{C}$. Then $U\in\mathcal{C}$ belongs to $\mathcal{C}_{\rm simp}$ if and only if it is indecomposable, and every proper non-zero subrepresentation contains a non-zero projective summand. 

If a vertex $i\in Q_0$ is not a sink, the simple representation $S_i$ belongs to $\mathcal{C}_{\rm simp}$. If $i$ is a sink, then $\tau^{-1}S_i$ belongs to $\mathcal{C}_{\rm simp}$. But in general, these examples do not exhaust $\mathcal{C}_{\rm simp}$. 
For example, let $Q=1\rightarrow 2\leftarrow 3\rightarrow 4$. Then the AR quiver has the shape displayed in Figure (\ref{ara4}), 
\begin{figure}
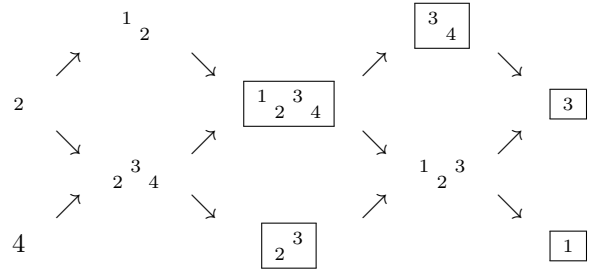
\label{ara4}
\[
\begin{array}{ccccccccc}
&&\begin{smallmatrix}
    1&\\&2
\end{smallmatrix}&&&&\boxed{\begin{smallmatrix}
    3&\\&4
\end{smallmatrix}}&&\\
&\nearrow&&\searrow&&\nearrow&&\searrow&\\
\begin{smallmatrix}
    2
\end{smallmatrix}&&&&\boxed{\begin{smallmatrix}
    1&&3&\\&2&&4
\end{smallmatrix}}&&&&\boxed{\begin{smallmatrix}
    3
\end{smallmatrix}}\\
&\searrow&&\nearrow&&\searrow&&\nearrow&\\
&&\begin{smallmatrix}
    &3&\\2&&4
\end{smallmatrix}&&&&\begin{smallmatrix}
    1&&3\\&2&
\end{smallmatrix}&&\\
&\nearrow&&\searrow&&\nearrow&&\searrow&\\
4&&&&\boxed{\begin{smallmatrix}
    &3\\2&
\end{smallmatrix}}&&&&\boxed{\begin{smallmatrix}
    1
\end{smallmatrix}}
\end{array}
\]

\caption{The AR-quiver of the quiver $1\rightarrow 2\leftarrow 3\rightarrow 4$. 
The vertices denote the support of the indecomposable $Q$-representations and the arrows are the irreducible maps between them. The highlighted vertices are the objects of $\mathcal{C}_{\rm simp}$}\label{Eq:ArQuiverA4}
\end{figure}
and the boxed representations are the relative simples in $\mathcal{C}$.

\begin{lem}\label{Lem:Almost projective} A representation $U$ belongs to $\mathcal{C}_{\rm simp}$ if and only if $U\in\mathcal{C}$ and every maximal subrepresentation of $U$ is projective.
\end{lem}

\begin{proof} Assume that $U$ belongs to $\mathcal{C}_{\rm simp}$, and that $U'\subset U$ is a proper subrepresentation. Then $U'$ contains a projective summand; write $U'=P\oplus U''$ for $U''\in\mathcal{C}$. If $U''\not=0$, we find an exact sequence
$$0\rightarrow U''\rightarrow U\rightarrow U/U''\rightarrow 0.$$
Then $U/U''$ also belongs to $\mathcal{C}$, and we arrive at a contradiction to $U$ being a relative simple. Thus $U''=0$, that is $U'$ is projective. Conversely, suppose 
that every maximal subrepresentation of $U\in\mathcal{S}$ is projective. Then an arbitrary 
subrepresentation of $U$, 
being contained in a maximal subrepresentation, is also projective.\end{proof}

Motivated by Lemma~\ref{Lem:Almost projective} we give the following definition. 
\begin{definition}
An element of $\mathcal{C}_{\rm simp}$ is called \emph{almost projective}.
\end{definition}

Recall that in a minimal projective resolution
\[
0\rightarrow P_1(V)\rightarrow P_0(V)\rightarrow V\rightarrow 0
\]
of a representation $V$, the multiplicity of a given projective indecomposable $P_i$ in $P_0(V)$ 
(resp.~$P_1(V)$) is given by $[V,S_i]$ (resp.~$[V,S_i]^1$).

\begin{rem}
It can be proved that a $Q$-representation is almost projective if and only if its minimal projective 
resolution is an irreducible homomorphism.  
\end{rem}
\begin{lem}\label{lemij} Suppose that $U\in\mathcal{C}_{\rm simp}$, that $P_i$ is a summand of 
$P_0(U)$, and that $P_j$ is a summand of $P_1(U)$. Then there exists an arrow $i\rightarrow j$, 
and the multiplicity of $P_j$ in $P_1(U)$ equals $1$. 
In particular, if $[U,S_j]^1\not=0$, then we have $[U,S_j]^1=1$.
\end{lem}

\begin{proof} 
If $P_i$ is a summand of $P_0(U)$, we have $[U,S_i]\not=0$, thus there exists an exact sequence
$$0\rightarrow P'\rightarrow U\rightarrow S_i\rightarrow 0$$
with $P'$ projective by the previous lemma. If $P_j$ is a summand of $P_1(U)$, we have $[U,S_j]^1\not=0$, and thus $[U,S_j]=0$ by directedness. We apply ${\rm Hom}(\_,S_j)$ to the previous sequence and get the exact sequence:
$$0\rightarrow {\rm Hom}(P',S_j)\rightarrow{\rm Ext}^1(S_i,S_j)\rightarrow{\rm Ext}^1(U,S_j)\rightarrow{\rm Ext}^1(P',S_j)=0.
$$
Thus $[S_i,S_j]^1\not=0$, which implies $[S_i,S_j]^1=1$ since this extension space has a basis indexed by the arrows $i\rightarrow j$. But this in turn implies $[U,S_j]^1=1$ (and $[P',S_j]=0$).\end{proof}

Dual to the above, we define $\mathcal{C}^\vee$ as the subcategory of representations without injective summands, or, equivalently, representations $M$ such that ${\rm Hom}(A^*,M)=0$,
where $A^*=\bigoplus_{j\in Q_0} I_j$. Dualizing the above lemmas, we find:

\begin{cor}\label{cordual} A representation $U$ belongs to $\mathcal{C}^\vee_{\rm simp}$ 
if and only if it belongs to $\mathcal{C}^\vee$ and every factor by a simple subrepresentation 
is already injective. For such a representation, $[S_i,U]^1\not=0$ already implies $[S_i,U]^1=1$.
\end{cor}

By one of the main results of Auslander-Reiten theory, the Auslander-Reiten translation $\tau$ 
induces an equivalence of categories
$$\tau:\mathcal{C}\stackrel{\sim}{\rightarrow}\mathcal{C}^\vee.$$
We also recall the Auslander-Reiten formula
$${\rm Ext}^1(V,W)\simeq{\rm Hom}(W,\tau V)^*\simeq{\rm Hom}(\tau^{-1}W,V)^*.$$

\begin{lem}\label{lemp0u} If $U\in\mathcal{C}_{\rm simp}$, then $P_0(U)$ is the direct 
sum of all $P_i$ such that $i$ is a source in the support of ${\rm\bf dim}(U)$.
\end{lem}

\begin{proof} 
Denote by $Q_0^I\subset Q_0$ the set of vertices $i$ such that $P_i$ is a summand of $P_0(U)$, 
equivalently, such that $[U,S_i]\not=0$. Denote by $Q_0^J\subset Q_0$ the set of vertices $j$ 
such that $P_j$ is a summand of $P_1(V)$, equivalently, such that $[U,S_j]^1\not=0$. Note that 
both subsets are non-empty since $U$ is not projective. The set $Q_0^I$ is clearly a subset 
of the support $Q_0^S\subset Q_0$ of ${\rm\bf dim}(U)$.
It necessarily has to contain all sources in $Q_0^S$ to ensure that the support of the 
dimension vector of $P_0(U)$ contains $Q_0^S$. By Lemma \ref{lemij}, there exists an arrow 
from every vertex in $Q_0^I$ to every vertex in $Q_0^J$.
Now assume there exists $i\in Q_0^I$ which is not a source in $S$. Then there exists a source 
$i'\in Q_0^S\subset Q_0^I$ and a path from $i'$ to $i$. On the other hand, for a vertex $j\in Q_0^J$ 
we find arrows $i\rightarrow j\leftarrow i'$, contradicting the fact that $Q$ is a tree. 
This proves that $Q_0^I$ consists precisely of the sources in $Q_0^S$. To prove that the $P_i$ 
for $i\in Q_0^I$ appear in $P_0(U)$ with multiplicity $1$, we use the Auslander-Reiten formula 
to see that $0\not=[U,S_i]=[S_i,\tau U]^1$, and $\tau U$ belongs to $\mathcal{C}^\vee_{\rm simp}$. 
By Corollary \ref{cordual}, we find  $1=[S_i,\tau U]^1=[U,S_i]$, yielding the claim.
\end{proof}

\begin{cor}\label{cornew} In the setting of Lemma \ref{lemij}, the multiplicity of $P_i$ in $P_0(U)$ equals $1$. Moreover, one of $P_0(U)$ or $P_1(U)$ is indecomposable, and the other one has at most three indecomposable direct summands.
\end{cor}

\begin{proof} In the terminology of the previous proof, we know that there exists an arrow between any vertex of $Q_0^I$ and any vertex of $Q_0^J$, But since $Q$ is Dynkin, this implies $|Q_0^I|\cdot|Q_0^J|\leq 3$, yielding the second claim. The first claim is already derived in the previous proof.\end{proof}


Denote by $P_{\rm min}$ the direct sum of all $P_i$ where $i$ is a source in $Q$. Dually, denote by $I_{\rm min}$ the direct sum of all $I_i$ where $i$ is a sink in $Q$. In other words, $P_{\rm min}$ (resp.~$I_{\rm min}$) is the minimal projective (resp.~injective) representation whose dimension vector has support $Q$.

\begin{cor}\label{corp0u} For every $U\in \mathcal{C}_{\rm simp}$, the projective cover $P_0(U)$ embeds into $P_{\rm min}$.
\end{cor}

\begin{proof} By the previous lemma, $P_0(U)$ is the direct sum of the $P_i$ for $i$ a source in the support $Q_0^S$ of ${\rm\bf dim}(U)$. Thus, the entry of ${\rm\bf dim}(P_0(U))$ at a vertex $j$ equals the number of sources in $Q_0^S$ admitting a path to $j$ (since there is at most one path between any two vertices in a Dynkin quiver). 
Similarly, the entry of ${\rm\bf dim}(P_{\rm min})$ at $j$ equals the number of sources in $Q$ admitting a path to $j$. This yields ${\rm\bf dim}(P_0(U))\leq{\rm\bf dim}(P_{\rm min})$. Choosing an arbitrary representation $T$ of dimension vector ${\rm\bf dim}(P_{\rm min})-{\rm\bf dim}(P_0(U))$, we have a degeneration $P_{\rm min}\leq_{\rm deg}P_0(U)\oplus T$, and thus $P_0(U)$ embeds into $P_{\rm min}$ by Proposition \ref{constremb}.\end{proof}

\section{The boundary representations}\label{sec:boundrep}

\begin{lem} Let $\mathcal{U}_{P,I}\subset R_{\bf d}(Q)$ be the open subset of representations degenerating to $P\oplus I$. Then the complement of $\mathcal{U}_{P,I}$ is the union of the closed subsets
$$\mathcal{A}_U=\{M\, |\, [U,M]>[U,I]\},$$
where $U$ ranges over $\mathcal{C}_{\rm simp}$. 
\end{lem}

\begin{proof} If $M$ does not degenerate to $P\oplus I$, then $[U,M]>[U,P\oplus I]$ for a representation $U$; choose $U$ of minimal dimension with this property. Then $U$ is indecomposable, and then necessarily non-projective, so $U\in\mathcal{C}$ and $[U,M]>[U,I]$. If $U$ is not simple in $\mathcal{C}$, it admits a subrepresentation $V\subset U$ such that $V,U/V\in\mathcal{C}$. By minimality of $U$, we then have $[V,M]\leq[V,I]$ and $[U/V,M]\leq[U/V,I]$, which implies $[U,M]\leq[U,I]$, a contradiction.
\end{proof}

\begin{lem} For $U\in\mathcal{C}_{\rm simp}$, the set $\mathcal{A}_U$ is irreducible.
\end{lem}

\begin{proof} For arbitrary $U$, the condition $[U,M]>a$ for a fixed $a$ can be interpreted as a rank condition. Namely, using a minimal projective resolution of $U$ as above, we find an exact sequence
$$0\rightarrow{\rm Hom}(U,M)\rightarrow{\rm Hom}(P_0(U),M)\stackrel{\Phi_M}{\rightarrow}{\rm Hom}(P_1(U),M),$$
thus, $[U,M]>a$ if and only if the rank of $\Phi_M$ is strictly less than $b=[P_0(U),M]-a$. 
In other words, considering the map 
\[
\Phi: R_{\bf d}(Q)\to {\rm Hom}({\rm Hom}(P_0(U),M),{\rm Hom}(P_1(U),M))
\]
given by $M\mapsto \Phi_M$, the set $\mathcal{A}_U$ is the inverse image of the set of linear maps of 
rank strictly less than $b$. The latter is always irreducible, but for general $U$, there is not much 
control about the fibres of $\Phi$. Now assume that $U\in\mathcal{C}_{\rm simp}$. By Lemma \ref{lemij}, 
we have $P_0(U)=\bigoplus_{i\in Q_0^I}P_i$ 
and $P_1(U)=\bigoplus_{j\in Q_0^J}P_j$ for subsets $Q_0^I,Q_0^J\subset Q_0$ such that there is an arrow 
from every vertex in $Q_0^I$ to every vertex in $Q_0^J$. 
This implies that the map $\Phi_M$ above is nothing else than a block matrix with entries being 
(scalar multiples of) the matrices $M_\alpha$ representing the arrows $\alpha$ between vertices in 
$Q_0^I$ and $Q_0^J$, respectively, in a representation $M$. In particular, the map $\Phi$ is linear. 
In this case, we can thus conclude that $\mathcal{A}_U$ is also irreducible.
\end{proof}

\begin{cor}\label{cormu} For every $U\in\mathcal{C}_{\rm simp}$, there exists a representation $M_U$ such that $[U,M]>[U,I]$ if and only if $M$ is a degeneration of $M_U$. 
\end{cor}

\begin{proof} The set $\mathcal{A}_U$ is anyway closed and stable under the structure group $G_{\bf d}$, 
and irreducible by the previous lemma, thus it equals the closure of the orbit of a representation $M_U$, 
which thus has the claimed property.
\end{proof}

The representation $M_U$ can be constructed rather naturally as a generic extension.

\begin{lem}\label{lemconstructionmu} For every $U\in\mathcal{C}_{\rm simp}$, the representation $M_U$ of the previous corollary fits into an exact sequence
$$0\rightarrow V_U\rightarrow M_U\rightarrow C_U\rightarrow 0,$$
where ${\rm Ext}^1(V_U,V_U)=0={\rm Ext}^1(C_U,C_U)$, the above exact sequence is generic and thus $M_U=C_U*V_U$, and $V_U$ is isomorphic to a factor of $U^{[U,I]+1}$ by a projective subrepresentation.\end{lem}

\begin{proof} We define $D=[U,I]+1$. Since $[U,M_U]>[U,I]$ by the previous corollary, we can find linear independent elements $f_1,\ldots,f_D$ in ${\rm Hom}(U,M_U)$, and we consider the map $$f=[f_1\ldots f_D]:U^D\rightarrow M_U.$$
It induces short exact sequences
\[
0\rightarrow{\rm Ker}(f)\rightarrow U^D\rightarrow {\rm Im}(f)\rightarrow 0\mbox{ and }0\rightarrow {\rm Im}(f)\rightarrow M_U\rightarrow{\rm Coker}(f)\rightarrow 0.
\]
In the induced exact diagram
\[
\xymatrix{
{\rm Hom}(U,{\rm Ker}(f))\ar@{^{(}->}[r]&{\rm Hom}(U,U^D)\ar^{f^*}[rr]\ar[dr]&&{\rm Hom}(U,M_U)\\ 
&&{\rm Hom}(U,{\rm Im}(f))\ar@{^(->}[ur]&&
}
\]
the map $f^*$ is injective by the construction of $f$, thus ${\rm Hom}(U,{\rm Ker}(f))=0$. On the other hand, since $U$ is almost projective, ${\rm Ker}(f)$ has to be of the form $U^k\oplus P'$ for a projective representation $P'$. We thus find ${\rm Ker}(f)=P'$, and also $[U,{\rm Im}(f)]\geq D$. Let $C_U$ be the exceptional representation of the same dimension vector as ${\rm Coker}(f)$, and let $M'$ be the generic extension of $C_U$ by ${\rm Im}(f)$. Since $C_U$ degenerates to ${\rm Coker}(f)$, the representation $M'$ degenerates to $M_U$. We then have
$$[U,M']\geq [U,{\rm Im}(f)]\geq D=[U,I]+1,$$
thus, $M'$ belongs to $\mathcal{A}_U$. By the previous corollary, we thus find that $M_U$ also degenerates to $M'$, proving that $M'$ is isomorphic to $M_U$. We thus have an exact sequence
$$0\rightarrow{\rm Im}(f)\rightarrow M_U\rightarrow C_U\rightarrow 0.$$
Now, let $V_U$ be the exceptional representation of the same dimension vector as ${\rm Im}(f)$, and let $W$ be the generic extension of $V_U$ by ${\rm Ker}(f)=P'$. Since $V_U$ degenerates to ${\rm Im}(f)$, the representation $W$ degenerates to $U^D$, which implies that $W$ is isomorphic to $U^D$ since $U^D$ is exceptional. We thus have an exact sequence
$$0\rightarrow P'\rightarrow U^D\rightarrow V_U\rightarrow 0,$$
and in particular we have $[U,V_U]\geq D$. Now, let $M''$ be the generic extension of $C_U$ by $V_U$. Since $V_U$ degenerates to ${\rm Im}(f)$, the representation $M''$ degenerates to $M_U$. Again, we find
$$[U,M'']\geq [U,V_U]\geq D=[U,I]+1,$$
and the same argument as above proves that $M''$ is isomorphic to $M_U$. This constructs the exact sequence
$$0\rightarrow V_U\rightarrow M_U\rightarrow C_U\rightarrow 0$$
with the claimed properties.
\end{proof}

\section{Two strata}\label{sec:twostrata}

For a representation $V$, denote by ${\rm gr}(V)$ the direct sum of all its Jordan-Hölder factors.

\begin{lem} If $I$ contains $I_{\rm min}$, we have $[U,I]<[U,{\rm gr}(I)]$ for every 
$U\in\mathcal{C}_{\rm simp}$.
\end{lem}

\begin{proof} We have
$$[U,I]=[U,I]-[U,I]^1=[U,{\rm gr}(I)]-[U,{\rm gr}(I)]^1,$$
and thus $[U,I]=[U,{\rm gr}(I)]$ if and only if $[U,{\rm gr}(I)]^1=0$. By assumption, 
${\rm gr}(I)$ contains all simples as direct summands, thus $[U,{\rm gr}(I)]^1=0$ implies 
that $U$ is projective, a contradiction.
\end{proof}

\begin{cor}\label{cors1} Assume that $I$ contains $I_{\rm min}$ as a direct summand. 
Then for every $U\in\mathcal{C}_{\rm simp}$ the representation $M_U$ admits $P$ as a subrepresentation. 
The stratum $\mathcal{S}_{[P]}$ has dimension $[P,I]$.\end{cor}

\begin{proof} 
We have $[U,P\oplus {\rm gr}(I)]=[U,{\rm gr}(I)]>[U,I]$, and thus $M_U$ 
degenerates to $P\oplus{\rm gr}(I)$ by Corollary \ref{cormu}. Since $[P,M_U]=[P,P\oplus {\rm gr}(I)]$ 
and $P$ embeds into $P\oplus {\rm gr}(I)$, it also embeds into $M_U$.  The previous equality also shows that
$$\dim\mathcal{S}_{[P]}=[P,M_U]-[P,P]=[P,I].$$
\end{proof}

The following lemma is trivial if $P$  contains $P_0(U)$ as a direct summand, since we can then choose $P'(U)=P_0(U)$. So, it is only relevant for small $P$.

\begin{lem} Assume that $P$ contains $P_{\rm min}$ as a direct summand. Then, for every $U\in\mathcal{C}_{\rm simp}$, there exists a direct summand $P'(U)$ of $P$ containing $P_0(U)$ as a subrepresentation, 
such that $${\rm Ext}^1(P'(U)/P_0(U),P)=0,$$
and $P'(U)/P_1(U)$ is indecomposable.
\end{lem}

\begin{proof} 
By Lemma \ref{lemp0u}, $P_0(U)$ is the direct sum of all $P_i$ for $i$ in $Q_0^I$, the set of sources in the support of ${\rm\bf dim}(U)$. We consider the set of vertices $Q_0$ with its partial order defined by the arrows (that is, $i\leq j$ if there exists a path from $i$ to $j$). Let $Q_0^i$ be the set of vertices admitting a path to $i$ for $i\in I$. We claim that there are no paths between any two different $Q_0^i$. Namely, again by Lemma \ref{lemij}, we know that for any two vertices $i,i'\in Q_0^I$ we find a vertex $j$ and arrows $i\rightarrow j\leftarrow i'$, a contradiction to $Q$ being a tree. Using this information, we can define a specific $P'$ as above. Namely, for every $i\in Q_0^I$, let $Q_0^{K,i}$ be the set of $\leq$-maximal vertices $k(i)\in Q_0^i$ such that $P_{k(i)}$ is a direct summand of $P$. It is nonempty by assumption on $P$. Then we define $P'(U)$ as the direct sum of all $P_k$ for $k$ in $Q_0^K=\bigcup_{i\in I}Q_0^{K,i}$. Since $P_t$ embeds into $P_u$ if $u\leq t$, we see that $P_0(U)$ embeds into $P'(U)$. By definition, $P'(U)$ is a direct summand of $P$. Now we define $R=P'(U)/P_0(U)$, and we claim that $[R,P]^1=0$. Namely, assume that $[R,P_l]^1\not=0$ and that $P_l$ is a direct summand of $P$. From the exact sequence
$$0\rightarrow P_0(U)\rightarrow P'(U)\rightarrow R\rightarrow 0$$
we get an exact sequence
$$
0={\rm Hom}(R,P_l)\rightarrow{\rm Hom}(P'(U),P_l)\rightarrow{\rm Hom}(P_0(U),P_l)\rightarrow{\rm Ext}^1(R,P_l)\not=0.
$$
Thus $[P_0(U),P_l]>[P'(U),P_l]$ which, by the above explicit description of $P_0(U)$ and $P'$, 
means that the number of paths from $l$ to a vertex in $Q_0^I$ is strictly bigger than the number 
of paths from $l$ to $Q_0^K$. Since there are no paths between the sets $Q_0^i$ for $i\in Q_0^I$, 
this implies that, for some $i\in Q_0^I$, there exists a path from $l$ to $i$, but no path from $l$ 
to any vertex in $Q_0^{K,i}$. But since $P_l$ is a direct summand of $P$, the definition of $Q_0^{K,i}$ implies that $l\leq k$ for some vertex in $Q_0^{K,i}$, a contradiction.

Now we consider the following commutative diagram with exact rows and columns:

$$\begin{array}{ccccccc}
&0&&0&&&\\ &\downarrow&&\downarrow&&&\\ &P_1(U)&=&P_1(U)&&&\\
&\downarrow&&\downarrow&&&\\
0\rightarrow&P_0(U)&\rightarrow&P'(U)&\rightarrow&R&\rightarrow 0\\
&\downarrow&&\downarrow&&||&\\
0\rightarrow&U&\rightarrow&V&\rightarrow&R&\rightarrow 0\\
&\downarrow&&\downarrow&&&\\ &0&&0&&&\end{array}$$
Our aim is to prove indecomposability of $V$. Since $[R,P]^1=0$ and $P'(U)$ is a direct summand of $P$, we have $[R,P'(U)]^1=0$, and thus $[R,V]^1=0$ by the right vertical exact sequence. We claim that we also have $[R,V]=0$. Recall from the previous lemma and Lemma \ref{lemij} that 
$$P'(U)=\bigoplus_{k\in Q_0^K}P_k,\; P_0(U)=\bigoplus_{i\in Q_0^I}P_i,\; P_1(U)=\bigoplus_{j\in Q_0^J}P_j,$$
where there are only paths from $Q_0^K$ to $Q_0^I$ and from $Q_0^I$ to $Q_0^J$. This implies that 
$$[P'(U),P_1(U)]=0=[P_0(U),P_1(U)].$$
We then find
\begin{eqnarray*}
[R,V]&=&[P'(U),V]-[P_0(U),V]=\\
&=&[P'(U),P'(U)]-[P'(U),P_1(U)]-[P_0(U),P_1(U)]+[P_0(U),P_1(U)]=\\
&=&[P'(U),P'(U)]-[P_0(U),P_1(U)]=\\
&=&[R,P'(U)]=0.\end{eqnarray*}
The last equality follows since $R$ does not admit a non-zero projective summand: such a summand would induce a direct summand of $P'(U)$ which does not intersect $P_0(U)\subset P'(U)$, in contradiction to the construction of $P'(U)$.
Now we claim that we have $[U,R]=0$. Since $U\in\mathcal{C}$, we have a commuting square of embeddings
$$\begin{array}{ccc}
{\rm Hom}(P_0(U),P_0(U))&\rightarrow&{\rm Hom}(P_0(U),P'(U))\\
\downarrow&&\downarrow\\
{\rm Hom}(P_1(U),P_0(U))&\rightarrow&{\rm Hom}(P_1(U),P'(U))\end{array}$$
and the claim is equivalent to this diagram being Cartesian. Using the explicit form of the three projectives, we can make all four spaces of maps explicit as direct sums of spaces of paths
$$\begin{array}{ccc}
\bigoplus_{i\in Q_0^I} \langle i\leadsto i\rangle&\rightarrow&\bigoplus_{i\in Q_0^I} \langle k\leadsto i\, |\, k\in Q_0^{K,i}\rangle\\
\downarrow&&\downarrow\\
\bigoplus_{i\in Q_0^I}\langle i\leadsto j\, |\, j\in Q_0^J\rangle&\rightarrow&\bigoplus_{i\in Q_0^I}\langle k\leadsto i\mapsto j\, |\, k\in Q_0^{K,i}, j\in Q_0^J\rangle\end{array}$$
Each direct summand corresponding to a vertex $i\in Q_0^I$ constitutes a square of the form
$$\begin{array}{ccc}\mathbb{C}&\rightarrow&Y\\ \downarrow&&\downarrow\\ X&\rightarrow&X\otimes Y\end{array}$$
where the maps are induced by fixed non-zero choices of elements $x\in X$, $y\in Y$. Such a diagram is evidently cartesian, proving the claim $[U,R]=0$. Now from the induced exact sequences
$$0={\rm Hom}(R,V)\rightarrow{\rm Hom}(V,V)\rightarrow{\rm Hom}(U,V)\rightarrow{\rm Ext}^1(R,V)=0$$
and
$$0\rightarrow{\rm Hom}(U,U)\rightarrow{\rm Hom}(U,V)\rightarrow{\rm Hom}(U,R)=0$$
we conclude
$$[V,V]=[U,V]=[U,U]=1$$
since $U$ is indecomposable, thus, $V$ is indecomposable.
\end{proof}

\begin{lem}\label{lems2} Assuming that $P$ contains $P_{\rm min}$ as a direct summand, write $P=P''\oplus P'(U)$ as provided by the previous lemma. Define 
\begin{equation}\label{eq:NU}
N_U=P''\oplus P_1(U)\oplus P'(U)/P_1(U).    
\end{equation}
Then $N_U$ embeds into $M_U$, and the stratum $\mathcal{S}_{[N_U]}$ has dimension $[P,I]$.
\end{lem}

\begin{proof} Consider again the commutative diagram with exact rows and columns:

$$\begin{array}{ccccccc}
&0&&0&&&\\ &\downarrow&&\downarrow&&&\\ &P_1(U)&=&P_1(U)&&&\\
&\downarrow&&\downarrow&&&\\
0\rightarrow&P_0(U)&\rightarrow&P'(U)&\rightarrow&R&\rightarrow 0\\
&\downarrow&&\downarrow&&||&\\
0\rightarrow&U&\rightarrow&V&\rightarrow&R&\rightarrow 0\\
&\downarrow&&\downarrow&&&\\ &0&&0&&&\end{array}$$

With this notation, we have $N_U=P''\oplus P_1(U)\oplus V$. From the previous lemma, we know that $$[R,P]^1=[R,V]=[R,V]^1=0$$ and that $V$ is indecomposable.  We then have
$$[U,N_U\oplus I]=[U,V]+[U,I]>[U,I],$$
thus, $M_U$ degenerates to $N_U\oplus I$ by Corollary \ref{cormu}. 
Applying ${\rm Hom}(\_,M_U)$ to the lower horizontal sequence, we find an exact sequence
$$0\rightarrow{\rm Hom}(R,M_U)\rightarrow {\rm Hom}(V,M_U)\rightarrow{\rm Hom}(U,M_U)\rightarrow {\rm Ext}^1(R,M_U).$$
This implies

\begin{eqnarray*}[V,M_U]&\geq& [R,M_U]-[R,M_U]^1+[U,M_U]\\
&=&[R,N_U\oplus I]-[R,N_U\oplus I]^1+[U,M_U]=\\
&=&[R,I]+[U,M_U]=\\
&=&[V,I]-[U,I]+[U,M_U]=\\
&>&[V,I].
\end{eqnarray*}

From this we conclude
\begin{eqnarray*}[N_U,M_U]-[N_U,N_U\oplus I]
&=&[P''\oplus P_1(U),M_U]-[P''\oplus P_1(U),N_U\oplus I]+\\
&&+[V,M_U]-[V,N_U\oplus I]=\\
&=&[V,M_U]-[V,V]-[V,I]\\
&\geq&1-[V,V]=0,\end{eqnarray*}
thus, $$[N_U,M_U]=[N_U,N_U\oplus I]$$ since $M_U$ degenerates to $N_U\oplus I$. 
The embedding of $N_U$ into $N_U\oplus I$ thus induces an embedding $N_U\subset M_U$, and we also conclude
$$[N_U,M_U]-[N_U,N_U]=[N_U,I]=[P,I].$$
\end{proof}

Here is the main result of our paper.

\begin{thm} If $P$ (resp.~$I$) contains $P_{\rm min}$ (resp.~$I_{\rm min})$ as a direct summand, then ${\rm Gr}_{\bf e}(M)$ is irreducible of dimension $[P,I]$ if and only if $M$ degenerates to $P\oplus I$.
\end{thm}

\begin{proof} In the universal quiver Grassmannians ${\rm Gr}_{\bf e}^Q({\bf d})$ over $R_{\bf d}(Q)$, 
the property of a fibre being irreducible of minimal dimension (which is equal to $[P,I]$) is open. 
Since this property holds for $P\oplus I$, it holds for all representations degenerating to $P\oplus I$. 
Conversely, we know that the complement of $\mathcal{U}_{P,I}$ is the union of the closures of the orbits of representations $M_U$ for $U$ ranging over $\mathcal{C}_{\rm simp}$ (see Corollary \ref{cormu}). 
By semicontinuity of fibre dimension, and by semicontinuity of the number of irreducible components of maximal dimension in a flat family \cite[Theorem2]{CFFFR17}, 
it suffices to prove that each ${\rm Gr}_{\bf e}(M_U)$ is either of dimension strictly larger than 
$[P,I]$, or is of this dimension but not irreducible. 
This is provided by Corollary \ref{cors1} and Lemma \ref{lems2}, which yield two different strata of dimension (at least) $[P,I]$.
\end{proof}



\begin{rem}\label{remquestions} The following two questions arise naturally:
\begin{enumerate}
\item Is is true that $\dim {\rm Gr}_{\bf e}(M_U)=[P,I]$?
\item Are the $\mathcal{A}_U$ for $U\in\mathcal{C}_{\rm simp}$ precisely the irreducible components of the complement of $\mathcal{U}_{P,I}$?
\end{enumerate}
At the moment we are not able to answer these questions. The following lemma gives a criterion for verifying that 
$\dim {\rm Gr}_{\bf e}(M) = [P,I]$.
\end{rem}

\begin{lem}\label{lemflatlocus} If $[U,M]\leq [U,I]+1$ for all non-projective indecomposable $U$, then ${\rm Gr}_{\bf e}(M)$ has dimension $[P,I]$.
\end{lem}

\begin{proof} Let $M$ be as in the claim, and suppose $N$ is a subrepresentation of $M$ of dimension vector ${\bf e}$. The claim follows once we prove that $[N,M]-[N,N]\leq [P,I]$. We decompose $N=N_P\oplus N'$, where $N_P$ is projective, and $N'\in\mathcal{C}$. Further decomposing
$$N'=U_1\oplus\ldots\oplus U_s$$
into indecomposable direct summands, we can estimate
$$[N',M]-[N',N']=\sum_i[U_i,M]-\sum_{i,j}[U_i,U_j]\leq \sum_i[U_i,I]+s-\sum_{i,j}[U_i,U_j]\leq [N',I].$$
This allows us to estimate
$$[N,M]-[N,N]=[N_P,M]-[N_P,N]+[N',M]-[N',N_P\oplus N']=$$
$$=[N_P,I]+[N',M]-[N',N']=[P,I]-[N',I]+[N',M]-[N',N']\leq [P,I],$$
as claimed.
\end{proof}


\section{Examples}\label{sec:examples}

\subsection{}
Let $Q$ be the equioriented type $A_n$ quiver with the vertices labeled
by $1,\dots,n$ and arrows $i\to i+1$. Let $P=\bigoplus_{i\in Q_0} P_i$
and let $I=\bigoplus_{j\in Q_0} I_j$. 
Then 
\[
\bd=\mathbf{dim}\, (P\oplus I) = (n+1,\dots,n+1),\ \bdim(P)=(1,\dots,n).
\]
The set $\mathcal{C}_{simp}$ consists of simple representations $S_i$, $1\le i\le n-1$,
since the last simple representation $S_n$ is projective and any other indecomposable
admits a subrepresentation without projective summands. 

Let $U=S_i$, $1\le i\le n-1$. Then $A_U$ consists of representations 
$M\in \mathrm{Rep}_{\bd}(Q)$ such that $[S_i,M]>1$; hence, $A_U$ consists of
representations $M$ which are degenerations of 
\[
M_{S_i} = P_1^{n-1}\oplus I_i^2\oplus P_{i+1}^2
\]
The representation $N_U$ \eqref{eq:NU} for $U=S_i$ is obtained from $P$ by replacing the summand
$P_i$ by $P_{i+1}\oplus S_i$ (since $P$ contains the projective cover $P_i$ of $S_i$
as a summand). The representations $N_{S_i}$ and $P$ admit embeddings into $M_{S_i}$, and 
the dimensions of the strata of subrepresentations isomorphic to either $P$ or 
$N_{S_i}$ are equal to $n(n+1)/2$.

We note that all $A_{S_i}$, $i=1,\dots,n-1$ contain the representation
\[
M^{(2)} = \bigoplus_{i\in Q_0} P_i \oplus \bigoplus_{i=1}^{n-1} I_i \oplus 
\bigoplus_{i\in Q_0} S_i
\]
The quiver Grassmannian $\Gr_{\rm\bf dim(P)}(M^{(2)})$ is called in \cite{CFFFR17} the ${\mathrm{mf}}$-degenerate flag variety
(${\mathrm{mf}}$ is for maximally flat). It was shown that $\Gr_{\rm\bf dim(P)}(M^{(2)})$
is reducible and equidimensional of the expected dimension 
$n(n+1)/2$ (with the Catalan number of irreducible
components). In particular, this implies that each quiver Grassmannian 
$\Gr_{\rm\bf dim(P)}(M_U)$, $U\in \mathcal{C}_{simp}$ is also of dimension 
$n(n+1)/2$.

\subsection{}
Let $Q$ be a $A_4$ quiver with the following orientation:
\[
\begin{tikzcd}[scale cd = 0.7]
1\arrow[r] & 2 & \arrow[l] 3\arrow[r] & 4. 
\end{tikzcd}	
\]
We fix $P=\mathbb{C}Q$ and $I=\mathbb{C}Q^\ast$, so  $\bd={\rm\bf dim}(P\oplus I)=(3,4,4,3)$ and 
$\mathbf{e}={\rm\bf dim}(P)=(1,3,1,2)$. Thus, $\langle{\bf e},{\bf d-e}\rangle=[P,I]=7$.

The set $\mathcal{C}_{simp}$ consists of five elements, which we denote in a standard way 
by $U(1)$, $U(3)$, $U(23)$, $U(34)$, and $U(14)$ (for example, 
$\dim U(14)=(1,1,1,1)$). Let us consider all five cases.

\underline{Let $U=U(1)=S_1$}. Then   $[U,I]=1$ and hence $A_U$ consist of $M\in R_\bd(Q)$ such that
$[S_1,M]\ge 2$. The most general representation satisfying this property is 
$M_U=U(14)\oplus U(1)^2\oplus U(24)^2\oplus U(23)$.
The projective resolution for $S_1$ is $0\to P_2\to P_1\to S_1\to 0$.
Since $P\supset P_1$, 
the subrepresentation $N_U$ is given by $N_U=U(1)\oplus U(2)^2\oplus U(24)\oplus U(4)$.

\underline{Let $U=U(3)=S_3$}. Then $[U,I]=1$ and hence $A_U$ consist of $M\in R_\bd(Q)$ such that
$[S_3,M]\ge 2$. The most general representation satisfying this property is 
$M_U=U(14)\oplus U(12)^2\oplus U(24)\oplus U(3)^2\oplus U(4)$.
The projective resolution for $S_3$ is $0\to S_2\oplus S_4\to P_3 \to S_3\to 0$.
Since $P\supset P_3$, the subrepresentation
$N_U$ is given by $N_U=U(12)\oplus U(2)^2\oplus U(3)\oplus U(4)^2$, which also embeds into
$M_U$ with $\dim S_{[N_U]}=7$.

\underline{Let $U=U(23)$}. Then $[U,I]=2$ and hence $A_U$ consist of $M\in R_\bd(Q)$ such that
$[U(23),M]\ge 3$. The most general representation satisfying this property is 
$M_U=U(14)\oplus U(13)^2\oplus U(23)\oplus U(4)^2$.
The projective resolution for $U(23)$ is $0\to P_4\to P_3 \to U(23)\to 0$.
Since $P\supset P_3$, the subrepresentation
$N_U$ is given by $N_U=U(12)\oplus U(2)\oplus U(23)\oplus U(4)^2$, which also embeds into
$M_U$.

\underline{Let $U=U(34)$}. Then $[U,I]=2$ and hence $A_U$ consist of $M\in R_\bd(Q)$ such that
$[U(34),M]\ge 3$. The most general representation satisfying this property is 
$M_U=U(12)^3\oplus U(34)^3\oplus U(23)$.
The projective resolution for $U(34)$ is $0\to P(2)\to P(3) \to U(34)\to 0$.
Since $P\supset P(3)$,  
$N_U$ is given by $N_U=U(12)\oplus U(2)^2\oplus U(34)\oplus U(4)$.

\underline{Let $U=U(14)$}. Then $[U,I]=4$ and hence $A_U$ consist of $M\in R_\bd(Q)$ such that
$[U(14),M]\ge 5$. The most general representation satisfying this property is 
$M_U=U(13)\oplus U(14)\oplus U(1)\oplus U(2)^2\oplus U(34)^2$.
The projective resolution for $U(14)$ is $0\to P_2\to P_1\oplus P_3 \to U(14)\to 0$.
Since $P\supset P_1\oplus P_3$, the subrepresentation 
$N_U$ is given by $N_U=U(14)\oplus U(2)^2\oplus U(4)$.

In all the cases above, the representations $P$ and $N_U$ embed into $M_U$ and 
the dimensions of the strata $S_{[P]}, S_{[N_U]}\subset \Gr_{{\rm\bf dim}(P)}(M_U)$ are equal to $7$.

Computer experiments show that there exists a unique most degenerate representation $M^2$
such that the quiver Grassmannian 
$\Gr_{{\rm \bf dim}(P)}(M^{(2)})$ is of expected dimension $[P,I]$ and any other $M\in R_\bd(Q)$ with this property
degenerates to $M^{(2)}$. The representation $M^{(2)}$ is explicitly given by 
\[
M^{(2)}=U(12)\oplus U(23)\oplus U(34)\oplus U(1)^2\oplus U(2)^2\oplus U(3)^2\oplus U(4)^2.
\]
One verifies that for any $U\in \mathcal{C}_{simp}$ the following inequality holds:
$[U,M^{(2)}]>[U,I]$. Hence $M^{(2)}$ belongs to $A_U$ for all $U$ and so
$\dim \Gr_{{\rm\bf dim}(P)}(M_U)=[P,I]$ for all $U\in \mathcal{C}_{simp}$.

\begin{example}
Let $Q$ be a $D_4$ quiver with the following orientation
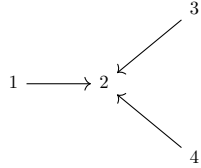
\[
\begin{tikzcd}[scale cd = 0.7]
& & 3 \arrow[dl] \\
1 \arrow[r] & 2 &  \\
& & 4 \arrow[ul]
\end{tikzcd}	
\]
We fix $P=\bigoplus_{i\in Q_0} P_i$, $I=\bigoplus_{j\in Q_0} I_j$, so  
$\bd={\rm\bf dim}(P\oplus I)=(3,5,3,3)$ and 
$\mathbf{e}={\rm\bf dim}(P)=(1,4,1,1)$ and $\langle{\bf e},{\bf d-e}\rangle=[P,I]=7$.

The set $\mathcal{C}_{simp}$ consists of seven (out of twelve) indecomposable
representations. Explicitly,
\[
\mathcal{C}_{simp}=\{U(1),U(3),U(4),U(123),U(124),U(234),U(12^234)\}
\]
(the last one is of total dimension five).
For the representations $U=U(1),U(3),U(4)$ the value of $[U,I]$ is equal to $1$, for
$U=U(12),U(23),U(24)$ the value of $[U,I]$ is equal to $3$, and
$[U(12^234),I]=5$.

Computer experiments show that in this case there exists a unique most degenerate representation $M^{(2)}$
with $\dim\Gr_\be(M^{(2)})=7$, which is explicitly given by
\[
M^{(2)} = U(1)^2\oplus U(3)^2\oplus U(4)^2\oplus U(2)^2\oplus U(12)\oplus
U(23)\oplus U(24).
\]
One verifies that for $U=U(1),U(3),U(4)$ one has $[U,M^{(2)}]=2$, 
for $U=U(12)$, $U(23)$, $U(24)$ one has $[U,M^{(2)}]=4$, and 
$[U(12^234),M^{(2)}]=6$. So $M^{(2)}\in A_U$ for all $U\in \mathcal{C}_{simp}$ and hence
$\dim {\rm Gr}_{\be}(M_U)=[P,I]=7$ for all $U\in \mathcal{C}_{simp}$.
\end{example}

\begin{example}
Let $Q$ be a $D_4$ quiver with the following orientation
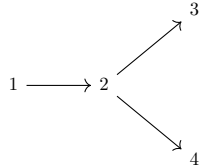
\[
\begin{tikzcd}[scale cd = 0.7]
& & 3  \\
1 \arrow[r] & 2\arrow[ur]\arrow[dr] &  \\
& & 4 
\end{tikzcd}	
\]
We fix $P=\bigoplus_{i\in Q_0} P_i$, $I=\bigoplus_{j\in Q_0} I_j$, so  
$\bd={\rm \bf dim}(P\oplus I)=(5,5,4,4)$ and 
$\mathbf{e}={\rm\bf dim}(P)P=(1,2,3,3)$ and $\langle{\bf e},{\bf d-e}\rangle=[P,I]=9$.

The set $\mathcal{C}_{simp}$ consists of six elements given by
\[
\mathcal{C}_{simp}=\{U(1),U(2),U(23),U(24),U(234),U(1234)\}.
\]
The values of $[U,I]$, $U\in \mathcal{C}_{simp}$ are given by $[U(1),I]=[U(2),I]=1$,
\[
[U(23),I]=[U(24),I]=2,\ [U(123),I]=3,\ [U(1234),I]=4.
\]
Computer experiments show that in this case there exist three most degenerate representations $M^{(2),a}$
with $\dim\Gr_\be(M^{(2)})=7$, which are explicitly given by
\begin{gather*}
M^{(2),1} = U(1)^2\oplus U(12)\oplus U(2)\oplus U(123)\oplus U(3)^2\oplus
U(24)\oplus U(4)^2 \oplus U(1234),\\
M^{(2),2} = U(1)^2\oplus U(12)\oplus U(2)\oplus U(23)\oplus U(3)^2\oplus
U(124)\oplus U(4)^2 \oplus U(1234),\\
M^{(2),3} = U(1)^2\oplus U(12)\oplus U(123)\oplus U(23)\oplus U(3)^2\oplus
U(124)\oplus U(24) \oplus U(4)^2.
\end{gather*}
One checks that for all $U\in \mathcal{C}_{simp}$ and
$a=1,2,3$ one has $[U,M^{(2),a}]>[U,I]$.
So $M^{(2),a}\in A_U$ for $a=1,2,3$ and for all $U\in \mathcal{C}_{simp}$, and hence
$\dim \Gr_{\bf}(M_U)=[P,I]=7$ for all $U\in \mathcal{C}_{simp}$.
\end{example}

\section{Prehomogeneity}\label{sec:prehom}
One might ask whether for $M^0$ the generic extension of a given $I$ by a given $P$, the group ${\rm Aut}(M^0)$ acts on ${\rm Gr}_{{\rm\bf dim}(P)}(M^0)$ with an open orbit. We translate this prehomogeneity question to another one which is more tractable.

For three representations $M$, $N$ and $X$, consider the variety $W(M,X,N)$ of exact sequences, given as the locally closed subset
$$\{(\alpha,\beta)\, |\, \alpha\mbox{ injective}, \,\beta\mbox{ surjective},\, \beta\alpha=0\}\subset{\rm Hom}(M,X)\times{\rm Hom}(X,N).$$

The group ${\rm Aut}(M)\times{\rm Aut}(X)\times{\rm Aut}(N)$ acts on $W(M,X,N)$ via
$$(\psi_M,\varphi,\psi_n)(\alpha,\beta)=(\varphi\alpha\psi_M^{-1},\psi_N\beta\varphi^{-1}).$$
The action of ${\rm Aut}(M)\times{\rm Aut}(N)$ is free, whereas the stabilizer of ${\rm Aut}(X)$ on $(\alpha,\beta)$ equals $1+\alpha{\rm Hom}(N,M)\beta$. Namely, if $\varphi\alpha=\alpha$ and $\beta\varphi=\beta$, then $(\varphi-1)\alpha=0$, thus $\varphi=1+h\beta$ for some $h\in{\rm Hom}(N,X)$. Then $\beta=\beta\varphi=\beta(1+h\beta)$, thus $\beta h\beta=0$, thus $\beta h$=0, thus $h=\alpha k$ for some $k\in{\rm Hom}(N,M)$.

Inside the vector space ${\rm Ext}^1(N,M)$, consider the subset $\mathcal{E}_{[X]}$ of extension classes $\zeta$ whose corresponding short exact sequence has middle term isomorphic to $X$. We claim that this is a locally closed subset. Namely, consider the linear map
$$\Phi_{M,N}:\bigoplus_{i\in Q_0}{\rm Hom}_\mathbb{C}(N_i,M_i)\rightarrow\bigoplus_{\alpha:i\rightarrow j}{\rm Hom}_\mathbb{C}(N_i,M_j)$$
given by
$$\Phi_{M,N}(f_i)=(M_\alpha f_i-f_jN_\alpha)_{\alpha:i\rightarrow j}.$$
Its cokernel is naturally isomorphic to ${\rm Ext}^1(N,M)$. Its target space 
naturally embeds into $R_{{\rm\bf dim}(X)}(Q)$ by mapping a tuple $(\zeta_\alpha)_\alpha$ to the point
$$\left[\begin{array}{cc}M_\alpha&\zeta_\alpha\\ 0&N_\alpha\end{array}\right].$$
Via these maps, the stratification of $R_{{\rm\bf dim}(X)}(Q)$ by orbits induces a stratification by locally closed subsets of ${\rm Ext}^1(N,M)$ by isomorphism class of the middle term of the associated exact sequence.

 We then find a well-defined map $W(M,X,N)\rightarrow\mathcal{E}_{[X]}$ given by associating to $(\alpha,\beta)$ the class of the corresponding exact sequence $$\zeta(\alpha,\beta)\, :\, 0\rightarrow M\stackrel{\alpha}{\rightarrow}X\stackrel{\beta}{\rightarrow}N\rightarrow 0.$$

By definition of ${\rm Ext}^1(N,M)$, this is a quotient map for the ${\rm Aut}(X)$-action, and it is also ${\rm Aut}(M)\times{\rm Aut}(N)$-equivariant.

On the other hand, consider the Grassmannian ${\rm Gr}_{{\rm\bf dim}(M)}(X)$ and its locally  closed subset $\mathcal{S}_{[M],[N]}$ of subrepresentations $U\subset X$ such that $U$ is isomorphic to $M$, with quotient $X/U$ isomorphic to $N$. Then we have a well-defined map $W(M,X,N)\rightarrow\mathcal{S}_{[M],[N]}$ associating to $(\alpha,\beta)$ the subrepresentation $U(\alpha,\beta)=\alpha(M)$. This is a quotient map for the ${\rm Aut}(M)\times{\rm Aut}(N)$-action, which is ${\rm Aut}(X)$-equivariant.

We thus obtain a diagram
$${\rm Gr}_{{\rm\bf dim}(M)}(X)\supset\mathcal{S}_{[M],[N]}\leftarrow W(M,X,N)\rightarrow \mathcal{E}_{[X]}\subset{\rm Ext}^1(N,M)$$
of quotient maps. This discussion shows the following lemma.

\begin{lem} The above diagram induces a bijection between ${\rm Aut}(X)$-orbits in $\mathcal{S}_{[M],[N]}$ and ${\rm Aut}(M)\times{\rm Aut}(N)$-orbits in $\mathcal{E}_{[X]}$ respecting orbit closures.
\end{lem}

In particular, ${\rm Aut}(X)$ admits an open orbit in $\mathcal{S}_{[M],[N]}$ if and only if ${\rm Aut}(M)\times{\rm Aut}(N)$ admits an open orbit in $\mathcal{S}_{[X]}$.

Now consider the particular case where $M=P$ is projective, $N=I$ is injective, and $M^0$ is the exceptional representation of dimension vector ${\bf d}={\rm\bf dim}(M)+{\rm\bf dim}(N)$. Then we know that ${\rm Gr}_{{\rm\bf dim}(P)}(M^0)$ has dimension $[P,I]$, that the stratum of subrepresentations isomorphic to $P$ has this dimension, and that, dually, the stratum of subrepresentations with quotient isomorphic to $I$ has this dimension. In the above language, we thus find that $\mathcal{S}_{[P],[I]}$ is a dense stratum in ${\rm Gr}_{{\rm\bf dim}(P)}(M^0)$. Thus ${\rm Aut}(M^0)$ admits an open orbit in this quiver Grassmannian if and only if it admits an open orbit in this stratum. On the other hand, by the definition of $\mathcal{E}_{[M^0]}\subset{\rm Ext}^1(I,P)$, we find that $\mathcal{E}_{[M^0]}$ is dense, and thus ${\rm Aut}(P)\times{\rm Aut}(I)$ admits an open orbit in ${\rm Ext}^1(I,P)$ if and only if it does so in $\mathcal{E}_{[M^0]}$. Combined with the previous lemma, we thus find:
\begin{cor}\label{cor:Ext(I,P)} 
${\rm Aut}(X)$ admits an open orbit in ${\rm Gr}_{{\rm\bf dim}(P)}(M^0)$ if and only if the group ${\rm Aut}(P)\times{\rm Aut}(I)$ admits an open orbit in the vector space ${\rm Ext}^1(I,P)$.
\end{cor}

Let us consider $Q$ given by $1\rightarrow 2\rightarrow\ldots\rightarrow n$. We would like to make the action of ${\rm Aut}(P)\times{\rm Aut}(I)$ on ${\rm Ext}^1(I,P)$ explicit. To do this, we consider the case of indecomposables first. We have ${\rm Hom}(P_i,P_{i'})\not=0$ if and only $i\geq i'$, and  choose a generator $\iota_{i,i'}$ in this case. Similarly, ${\rm Hom}(I_j,I_{j'})\not=0$ if and only if $j\geq j'$, and we choose a generator $\pi_{j,j'}$. The space ${\rm Ext}^1(I_j,P_i)$ is one-dimensional if and only if $1<i\leq j+1\leq n$; let $\zeta_{i,j}$ be a generator in this case, which corresponds to a short exact sequence
$$0\rightarrow P_i\rightarrow P_1\oplus U_{i,j}\rightarrow I_j\rightarrow 0.$$ We then have
$$\iota_{i,i'}\zeta_{i,j}=\zeta_{i',j}$$
for $1<i'\leq i\leq j+1\leq n$ and
$$\zeta_{i,j}\pi_{j',j}=\zeta_{i,j'}$$
for $1\leq i-1\leq j\leq j'<n$.

Now assume $P=\bigoplus_iP_i^{a_i}$ and $I=\bigoplus_iI_i^{b_i}$. Then ${\rm Aut}(P)$ can be identified with the parabolic of block upper triangular matrices with block sizes $a_1,\ldots,a_n$, and ${\rm Aut}(I)$ can be identified with the parabolic of block upper triangular matrices with block sizes $b_1,\ldots,b_n$. The space ${\rm Ext}^1(I,P)$ can be identified with block matrices having a matrix of size $a_i\times b_j$ as $(i,j)$-th block if $1<i\leq j+1\leq n$, and zero block otherwise. We then see that ${\rm Aut}(P)$ acts on ${\rm Ext}^1(I,P)$ through its quotient ${\rm Aut}(\bigoplus_{i>1}P_i^{a_i})$, and ${\rm Aut}(I)$ acts on ${\rm Ext}^1(I,P)$ through its quotient ${\rm Aut}(\bigoplus_{i<n}I_i^{b_i})$, and this action is given by the left and right action of upper triangular block matrices

$${\rm Aut}(\bigoplus_{i>1}P_i^{a_i})=\left[\begin{array}{ccc}{\rm GL}_{a_2}(\mathbb{C}){\rm id}_{P_1}&\ldots&M_{a_2\times a_n}(\mathbb{C})\iota_{n,2}\\
\vdots&&\vdots\\ 0&\ldots&{\rm GL}_{a_n}(\mathbb{C}){\rm id}_{P_n}\end{array}\right],$$
respectively
$${\rm Aut}(\bigoplus_{i<n}I_i^{b_i})=\left[\begin{array}{ccc}{\rm GL}_{b_1}(\mathbb{C}){\rm id}_{I_1}&\ldots&M_{b_1\times b_{n-1}}(\mathbb{C})\pi_{n-1,1}\\
\vdots&&\vdots\\ 0&\ldots&{\rm GL}_{b_{n-1}}(\mathbb{C}){\rm id}_{I_{n-1}}\end{array}\right],$$
on
$${\rm Ext}^1(I,P)=\left[\begin{array}{ccc}M_{a_2\times b_1}(\mathbb{C})\zeta_{2,1}&\ldots& M_{a_2\times b_{n-1}}(\mathbb{C})\zeta_{2,n-1}\\ \vdots&&\vdots\\ 0&\ldots&M_{a_n\times b_{n-1}}(\mathbb{C})\zeta_{n,n-1}\end{array}\right],$$
which indeed admits an open orbit (since even for the action via left and right multiplication of such parabolics on all matrices, we have finitely many orbits, thus an open one). We have thus proved using the following corollary:

\begin{thm} If $Q$ is a linearly oriented type $A_n$ quiver, then ${\rm Aut}(X)$ admits an open orbit in ${\rm Gr}_{{\rm\bf dim}(P)}(X)$ for all $P$ and $I$.
\end{thm}

\begin{example} From Corollary \ref{cor:Ext(I,P)} we see that $[P,P]+[I,I]\geq [I,P]^1$ is clearly a necessary condition for prehomogeneity. If $Q$ is not of linearly oriented type $A_n$, then $[I,P]=0$ for all $P$ and $I$, and thus we arrive at the necessary condition 
\[
\langle{\rm\bf dim}(P),{\rm\bf dim}(P)\rangle+\langle{\rm\bf dim} (I),{\rm\bf dim}(I)\rangle+\langle{\rm\bf dim}(I),{\rm\bf dim} (P)\rangle\geq 0.
\]
It is easy to see that this inequality already fails for $P=A$, $I=A^*$ if $Q$ is of type $A_5$ or $D_5$ with alternating orientation.
\end{example}


\begin{thebibliography}{99}

\bibitem[Bo96]{Bo96}
K.~Bongartz, On Degenerations and Extensions of Finite Dimensional Modules, Adv. Math.
121 (1996), 245–287.


\bibitem[CB92]{CB92} W.~Crawley-Boevey, \emph{Lectures on representations of quivers}. Preprint 1992. Available at the author's webpage. 


\bibitem[CR00]{CR00}
P.~Caldero and M.~Reineke, On the quiver {G}rassmannian in the acyclic
case. \emph{J. Pure Appl. Algebra} \textbf{212} (2008), no.~11, 2369--2380.

\bibitem[CI20]{CI20}
G~Cerulli~Irelli, \emph{Three lectures on quiver Grassmannians}, Contemporary Mathematics \textbf{758}, American mathematical Society, 2020.



\bibitem[CFFFR17]{CFFFR17}
G.~Cerulli Irelli, X.~Fang, E.~ Feigin, G.~Fourier, M.~Reineke, Linear degeneration of flag varieties. \emph{Math.~Z.} \textbf{287} (2017), no. 1-2, 615--654. 

\bibitem[CFFFR20]{CFFFR20}
G.~Cerulli Irelli, X.~Fang, E.~ Feigin, G.~Fourier, M.~Reineke, 
Linear degenerations of flag varieties: partial flags, defining equations, and group actions. 
\emph{Math. Z.} \textbf{296} (2020), no. 1-2, 453--477.

\bibitem[CFR12]{CFR12} G.~Cerulli Irelli, E.~ Feigin, M.~Reineke, Quiver Grassmannians and degenerate flag varieties. 
\emph{Algebra \& Number Theory} \textbf{6} (2012), no.~1, 165--194. 

\bibitem[CFR13-1]{CFR13-1} G.~Cerulli Irelli, E.~Feigin, M.~Reineke, Degenerate flag varieties: moment graphs 
and Schr\"oder numbers, \emph{J. Algebraic Combin.} \textbf{38} (2013), no.~1.


\bibitem[Fed25]{Fed25} 
S.~Fedotov,
Quiver representation library, \href{https://github.com/st-fedotov/quiver}
{https://github.com/st-fedotov/quiver}.

\bibitem[FedFe25]{FedFe25}
S.~Fedotov, E.~Feigin, \emph{PrIncipal quiver Grassmannians: conjectures}, https://arxiv.org/abs/2512.09731.

\bibitem[Fe11]{Fe11} 
E.~Feigin,  Degenerate flag varieties and the median Genocchi numbers.
\emph{Mathematical Research Letters} \textbf{18} (2011), no. 6, 1--16.


\bibitem[Fe12]{Fe12}
{E.~Feigin}, ${\mathbb G}_a^M$ degeneration of flag varieties.
\emph{Selecta Mathematica} \textbf{18}:3 (2012), 513--537.


\bibitem[Fe23]{Fe23} 
E.~Feigin, 
PBW degenerations, quiver Grassmannians, and toric varieties,
2023, International Congress of Mathematicians (vol. 4), EMS Press, Berlin, p. 2930--2946.


\bibitem[FFR17]{FFR17}
E.~Feigin, M.~Finkelberg, M.~Reineke, 
Degenerate affine Grassmannians and loop quivers. 
\emph{Kyoto J. Math.} \textbf{57} (2017), no.~2, 445--474. 


\bibitem[LW19]{LW19}
O.~Lorscheid, T.~Weist, Pl\"ucker relations for quiver Grassmannians. \emph{Algebr. Represent. Theory} \textbf{22} (2019), no.~1, 211--218.



\bibitem[Re13]{Re13}
M.~Reineke, Every projective variety is a quiver Grassmannian. \emph{Algebr. Represent. Theory} \textbf{16} (2013), 1313--1314. 

\bibitem[Re01]{Re01}
M.~Reineke, Generic extensions and multiplicative bases of quantum groups at q=0. \emph{Represent. Theory} \textbf{5} (2001), 147--163. 


\bibitem[Ri18]{Ri18}
C.~M.~Ringel, Quiver Grassmannians for wild acyclic quivers.  \emph{Proc.~AMS} \textbf{146} (2018), no.~5, 1873--1877. 

\bibitem[Schi14]{Schi14} 
R.~Schiffler, Quiver Representations, CMS Books in Mathematics, Springer-Verlag, 
Switzerland, 2014.




\end{thebibliography}
\end{document}